\documentclass[10pt,a4paper,twoside,reqno]{amsart}

\usepackage{amssymb}
\usepackage{mathtools}
\usepackage[shortlabels]{enumitem}
\usepackage[utf8]{inputenc}
\usepackage{color}
\usepackage{xspace}
\usepackage{microtype,mathabx}
\usepackage{fouridx}
\usepackage{tikz}
\usetikzlibrary{positioning,decorations.pathreplacing,decorations.markings,arrows.meta,calc,fit,quotes,cd,math,arrows,backgrounds,shapes.geometric}
\usepackage[a4paper,top=3cm,bottom=3cm,inner=2.5cm,outer=2.5cm]{geometry}
\usepackage{graphicx}
\definecolor{darkgreen}{rgb}{0,0.45,0}
\usepackage{bm}
\usepackage[colorlinks,citecolor=darkgreen,urlcolor=gray,final,hyperindex,linktoc=page,hyperfootnotes=true]{hyperref}
\usepackage[arrow, matrix, tips, curve, graph, rotate]{xy}
\SelectTips{cm}{10}

\usepackage{babel}
\usepackage{soul}

\setlist[enumerate]{label=(\roman*),itemsep=0ex}

\makeatletter

\renewcommand{\section}{\@startsection
{section}
{0}
{0mm}
{-\baselineskip}
{1\baselineskip}
{\centering \Large \bfseries}}

\renewcommand{\subsection}{\@startsection
{subsection}
{1}
{0mm}
{-\baselineskip}
{2mm}
{\large \bfseries}}

\makeatother

\usepackage[capitalize]{cleveref}
\crefname{equation}{}{}
\crefname{lem}{Lemma}{Lemmas}
\crefname{thm}{Theorem}{Theorems}
\crefname{defn}{Definition}{Definitions}
\crefname{conj}{Conjecture}{Conjectures}
\crefname{ex}{Example}{Examples}
\crefname{sec}{Section}{Sections}
\crefname{prop}{Proposition}{Propositions}

\theoremstyle{plain}
\newtheorem{thm}{Theorem}[section]
\newtheorem*{thm*}{Theorem}
\newtheorem{cor}[thm]{Corollary}
\newtheorem{lem}[thm]{Lemma}
\newtheorem{prop}[thm]{Proposition}

\theoremstyle{remark}

\theoremstyle{definition}
\newtheorem{defn}[thm]{Definition}
\newtheorem{ex}[thm]{Example}

\definecolor{mypurple}{rgb}{0.5, 0.0, 0.5}

\newcommand{\myemph}{\emph}
\newcommand{\ie}{i.e.\xspace}

\newcommand{\eg}{e.g.\xspace}
\newcommand{\viz}{viz.\xspace}

\newcommand{\defeq}{=_\mathrm{def}}
\newcommand{\co}{\colon}
\newcommand{\op}{\mathrm{op}}

\newcommand{\id}{\mathrm{id}}
\newcommand{\Id}{\mathrm{Id}}

\newcommand{\one}{\mathrm{I}}

\newcommand{\comp}{}
\newcommand{\Nat}{\mathbb{N}}

\newcommand{\copr}{+}
\newcommand{\coexp}{\rhd}

\newcommand\univu{\kappa}

\newcommand{\cat}{}
\newcommand{\catA}{\cat{A}}
\newcommand{\catB}{\cat{B}}
\newcommand{\catC}{\cat{C}}

\newcommand{\moncat}{}
\newcommand{\moncatM}{\moncat{M}}

\newcommand{\coc}{\mathit}
\newcommand{\cocC}{\coc{C}}
\newcommand{\cocD}{\coc{D}}
\newcommand{\cocE}{\coc{E}}

\newcommand{\ccat}{\mathcal}
\newcommand{\ccatK}{\ccat{K}}
\newcommand{\ccatL}{\ccat{L}}

\newcommand{\rig}{}
\newcommand{\rigR}{\rig{R}}
\newcommand{\rigS}{\rig{S}}
\newcommand{\rigT}{\rig{T}}

\newcommand{\Bim}{\mathsf{Bim}}

\newcommand{\psh}[1]{\mathsf{P}{#1}}

\newcommand{\emb}{\mathsf{y}}

\newcommand{\yon}{\emb}

\newcommand{\spsh}{\mathsf{P}}
\newcommand{\spshA}{\spsh(\catA)}

\newcommand{\spshM}{\spsh(\moncatM)}

\newcommand{\freesmc}{\mathsf{S}}

\newcommand{\Sym}{\mathsf{S}}

\newcommand\dual{^\vee}
\newcommand\sfU{\mathsf{U}}
\newcommand\sfUin{\mathsf{V}}

\newcommand\PSh[1]{\psh\!\left(#1\right)}

\DeclareMathOperator*{\colim}{colim}

\newcommand{\Env}{\mathsf{Env}}

\newcommand{\End}{\mathsf{End}}

\newcommand{\nc}{\mathsf}
\newcommand{\Set}{\nc{Set}}
\newcommand{\SET}{\nc{SET}}

\newcommand{\Cat}{\nc{Cat}}

\newcommand{\CAT}{\nc{CAT}}

\newcommand{\COC}{\nc{COC}}

\newcommand\Pres{\mathsf{Pres}}

\newcommand{\PsSMon}{\mathsf{SMon}}

\newcommand{\SMONCAT}{\nc{SMCAT}}
\newcommand{\SMCAT}{\SMONCAT}

\newcommand{\RIG}{\nc{RIG}}
\newcommand\Rig{\nc{Rig}}
\newcommand\FRIG{\RIG^\mathsf{Free}}
\newcommand\FRig{\Rig^{\nc{Free}}}

\newcommand\CRIG{\RIG^\mathsf{Conv}}
\newcommand\CRig{\Rig^\mathsf{Conv}}

\newcommand{\OpdRIG}{\RIG^{\nc{Opd}}}
\newcommand\OpdRig{\Rig^{\nc{Opd}}}

\newcommand{\OPD}{\nc{OPD}}

\newcommand\map[2]{\left[#1,#2\right]}

\newcommand\CATexp[2]{\hom\CAT{#1}{#2}}
\renewcommand\hom[3]{#1\!\left(#2,#3\right)}
\newcommand\inthom[3]{\underline{#1}\!\left(#2,#3\right)}
\newcommand\Hom[3]{#1\big(#2,#3\big)}

\tikzset{tick/.style={postaction={decorate,decoration={markings,mark=at
position 0.5 with {\draw[-] (0,.4ex) -- (0,-.4ex);}}}}}
\tikzset{bigtick/.style={postaction={decorate,decoration={markings,mark=at
position 0.5 with {\draw[-] (0,.6ex) -- (0,-.6ex);}}}}}

\newcommand{\sbul}{\scriptstyle\bullet}
\tikzset{bul/.style={postaction={decoration={markings,mark=at position 0.5 with
{\node{$\sbul$};}},decorate}}}

\tikzset{Rightarrow/.style={double equal sign distance,>={Implies},->},
triple/.style={-,preaction={draw,Rightarrow}}}

\newcommand{\hide}[1]{}

\newcommand{\stmap}{\ensuremath{\mathrm{st}}}
\renewcommand{\mathrlap}[1]{\enspace{#1}}
\newcommand{\tensor}{\cdot}
\newcommand{\initialcat}{\mathsf{0}}
\newcommand{\terminalcat}{\mathsf{1}}
\newcommand{\zeroobj}{\mathsf{O}}
\newcommand{\fullsubcat}{\hookrightarrow}
\newcommand{\embto}{\hookrightarrow}
\newcommand{\mylift}[1]{\overline{#1}}
\DeclareMathOperator*{\bilim}{bilim}
\DeclareMathOperator*{\bicolim}{bicolim}

\begin{document}

\title{Symmetric 2-Rigs: Coexponentiability and Cartesian Closure}

\author[M.\,Anel]{Mathieu Anel} 
\address{Laboratoire J.-A. Dieudonn\'e, Universit\'e C\^ote d'Azur, Nice,
  France} 
\email{mathieu.anel@protonmail.com}

\author[M.\,Fiore]{Marcelo Fiore}
\address{Department of Computer Science and Technology, University of
  Cambridge, Cambridge, United Kingdom} 
\email{marcelo.fiore@cl.cam.ac.uk}

\author[N.\,Gambino]{Nicola Gambino}
\address{Department of Mathematics, The University of Manchester, Manchester,
  United Kingdom} 
\email{nicola.gambino@manchester.ac.uk}

\keywords{Symmetric 2-rigs, exponentiability, cartesian closed structure.}

\subjclass[2020]{18N10, 18M60, 18M05, 18D60}

\begin{abstract}
We study coexponentiability in the context of the cocartesian 2-category 
$\RIG$ of symmetric \mbox{2-rigs}, symmetric strong monoidal cocontinuous
functors, and symmetric monoidal natural transformations.
Our results characterize the coexponentiable symmetric 2-rigs as those that
are deformation retracts of presheaf categories over small categories.  
As an application, we give an account of the cartesian closure of two full
sub-2-categories of the dual of $\RIG$ arising from the theory of
combinatorial species and the theory of symmetric operads.  
\end{abstract}

\date{\today}

\maketitle

\setcounter{tocdepth}{1}
\tableofcontents

\section{Introduction}

The idea of categorification~\cite{BaezJ:cat,MazorchukV:lecac} involves
replacing standard set-based mathematical structures (such as monoids) with
corresponding category-based structures (such as monoidal categories). Over
the years, this has led to significant advances in various mathematical
contexts, such as algebra, algebraic topology, and mathematical physics.  
This work considers symmetric 2-rigs, a categorification of the notion of a
commutative rig (a ring without negatives).  These appear naturally in many
areas of mathematics and have been studied widely in recent years; for
instance, in algebraic geometry (via categories of coherent
sheaves~\cite{BrandenburgM:refdca,ChirvasituA:funpga}), category
theory~\cite{BaezJ:higda,Baez:schfcp,LoregianF:dif2r}, and structural
combinatorics (via species of structures~\cite{JoyalA:thecsf,JoyalA:fonaes}). 

A symmetric 2-rig is a cocomplete category equipped with a symmetric monoidal
structure in which the tensor product is separately cocontinuous
(\ie~cocontinuous in each argument).  Here, the colimits play the role of the
additive structure, the symmetric monoidal structure that of the commutative
multiplicative structure, and the preservation of colimits that of the
distributive law between them.
Symmetric 2-rigs, symmetric strong monoidal cocontinuous functors, and
symmetric monoidal natural transformations assemble into a 2-category $\RIG$
that shares many similarities with the category of commutative rings.  In
particular, and of key importance to us here, it is (bicategorically)
cocartesian: the initial symmetric 2-rig is the cartesian category $\Set$ of
small sets and functions, and the coproduct of two symmetric 2-rigs $R$ and
$S$ is given by the tensor product $R\otimes S$ of their underlying cocomplete
categories (defined as the classifier $R \times S \to R\otimes  S$ of
separately-cocontinuous functors) equipped with the natural symmetric monoidal
structure (\cref{prop:RIGcocartesian}).
This paper investigates these mathematical structures with two central
motivations: firstly, to lay foundations for such categorification of
commutative algebra establishing results analogous to well-known fundamental
ones for commutative rings; secondly, to provide a conceptual understanding of
the exponential structure of two important full sub-2-categories of the dual
of $\RIG$.

Our main result proves a conjecture of Andr\'e Joyal~\cite{JoyalA:polyfun} 
characterizing the symmetric 2-rigs $R$ that are coexponentiable, \ie~such
that the coproduct pseudofunctor $R\otimes(-) \co \RIG \to \RIG$ admits a left
biadjoint.  Indeed, we prove that a symmetric 2-rig  is coexponentiable if and
only if its underlying cocomplete category is (bicategorically) dualizable in
the \mbox{2-category} $\COC$ of cocomplete categories, cocontinuous functors,
and natural transformations~(\cref{thm:coexponentiability-in-crig}).  
This, in combination with the characterization of dualizable cocomplete
categories as the deformation retracts of presheaf categories over small
categories~(\cref{thm:dualizable-coc}), provides a simple criterion for
establishing coexponentiability of symmetric 2-rigs as follows.

\begin{thm*} 
A symmetric 2-rig is coexponentiable if and only if its underlying cocomplete
category is a deformation retract of a presheaf category over a small category.
\end{thm*}

From the perspective of commutative algebra, our characterization of
coexponentiable symmetric 2-rigs may be seen as a 2-dimensional generalization
of a result for commutative rigs established by Niefield and
Wood~\cite{NiefieldS:coepr}.  
To note is that the categorical proof therein crucially relies on the
presentation of a commutative monoid as a coequalizer of a parallel pair
between coproducts (direct sums) of copies of the monoidal unit (additive
monoid of natural numbers).  Here instead, to circumvent the fact that this
structural property does not scale up to cocomplete categories, our proof
embodies a novel version of the method therein using the pseudomonadicity of
the underlying 2-functor $\mylift\sfU\co\RIG\to\COC$, with left biadjoint
$\mylift\Sym \co \COC \to \RIG$ sending a cocomplete category $C$ to what we
call its symmetric algebra 2-rig $\mylift\Sym\,C$ in analogy with the
situation in commutative algebra~\cite{EisenbudCommAlg} and topos
theory~\cite{BungeM:symt}~(\cref{thm:sym-alg-coc}).  

In preparation for the proof of the theorem, we also prove some results of
independent interest, including the pseudomonadicity of the 2-category $\RIG$
over the 2-category $\CAT$ of categories, functors, and natural
transformations~(\cref{prop:pseudomonadicitysquare}), providing extra evidence
for the analogy between commutative rigs and symmetric 2-rigs.
Additionally, and as a first step towards a categorification of differential
algebra encompassing the differentiation theory of analytic
functors~\cite{FioreM:monbdl}, we introduce the symmetric 2-rig of dual sets
$\Set[\varepsilon]$ (\cref{CoexpRigsExamples}~\ref{DualSets}).  This is shown
to arise as a categorification of the commutative rig $\Nat[\varepsilon]$ of
dual numbers over the commutative rig of natural numbers $\Nat$ and seen to be
coexponentiable as an immediate consequence of our theorem above.

We furthermore obtain new results on free 2-rigs (on small categories) and on
more general operadic \mbox{2-rigs}~\cite{AnelM:ope2r}.  In particular, we
show that free 2-rigs are a coexponential ideal of operadic 
2-rigs (\cref{prop:MixedCoexpFormula}).
This is then used to provide a uniform account of the (bicategorical)
cartesian closed structure of two full sub-2-categories of the dual of $\RIG$
(\cref{cor:CC-FRIG} and \cref{prop:CC-OpRIG}) equivalent to those established
in~\cite{FioreM:carcbg} and~\cite{GambinoN:opebaf}.

Finally, let us mention that we are confident that our results will carry over
to the higher categorical setting with analogous proofs.

\subsubsection*{Outline of the paper} 
\cref{sec:background} reviews the bicategorical background needed in the
paper.  
Sections~\ref{sec:coc} and~\ref{sec:sym2rig}  establish our
results concerning $\COC$ and $\RIG$, respectively.
\cref{sec:carcbs} concludes discussing cartesian closed structure within
$\RIG^\op$. 

\subsubsection*{Acknowledgements}

Mathieu Anel acknowledges that the research leading to these results has
received funding from the European Research Council (ERC) under the European
Union's Ninth Framework Programme Horizon Europe (ERC Synergy Project Malinca,
Grant Agreement n.~101167526). 
Marcelo Fiore acknowledges that this material is based upon work supported by
EPSRC via grant EP/V002309/1.  
Nicola Gambino acknowledges that this material is based upon work supported by
the US Air Force Office for Scientific Research under award number
FA9550-21-1-0007, by EPSRC via grant EP/V002325/2, and ARIA via grant
MSAI-PR01-P12. 

\section{Preliminaries}
\label{sec:background}

\subsection*{Categories} 
\label{subsec:SizeConventions} 

We will mainly consider categories, notably without assuming the locally
smallness condition, and also those that are complete with respect to small
colimits.  
So as to carefully deal with size issues, throughout the paper we assume two
regular cardinals $\lambda>\kappa>\aleph_0$.
We will therefore deal with sets of four possible increasing sizes: 
finite sets (\ie~of cardinality less than $\aleph_0$); 
small sets (\ie~of cardinality less than $\kappa$); 
sets~(\ie~of cardinality less than $\lambda$); and 
large sets (\ie~of any cardinality). 

A (finite, small, large) category $K$ has a (finite, small, large) set of
isomorphism classes of objects and a (finite, small, large) hom-set of maps
between them, written $\hom KXY$, for every pair of objects $X, Y \in K$. The
composition of maps $f \co X \to Y$ and $g \co Y \to Z$ is written 
$g f \co X \to Z$, while the identity map on an object $X \in K$ is written
$\id_X \co X \to X$.  

We write $\SET$ for the large category of sets and functions, and  $\Set$ for
its full subcategory spanned by small sets.  Analogously, we write $\CAT$ for
the large category of categories and functors, and $\Cat$ for its full
sub-category spanned by small categories.  Note that $\Set$ (resp.~$\Cat$) is
dense in $\SET$ (resp.~$\CAT$).  Note also that $\Set$ is a category, and
therefore  $\Set \in \CAT$, but it is not a small category.  The situation is
summarised in \cref{tab:size}.
The large category $\CAT$ has cartesian closed structure denoted
$\big({(-)\!\times\!(=)},\terminalcat,\CATexp-=\big)$ and cocartesian
structure denoted $\big({(-)+(=)},\initialcat)$.
\begin{table}[htb]
\begin{tabular}{|c|c|c|c|} \hline
 & isomorphism-classes of objects & hom-sets of maps & Examples 
\\[.5mm] \hline
small category & 
small set  & 
small set & $\initialcat$, $\terminalcat$  
\\[.5mm]
locally small category  & 
set &  
small set &  $\Set$, $\Cat$ 
\\[.5mm]
category  & 
set &  
set & $\Set^\Cat$
\\[.5mm]
large category & 
large set &  
large set &  $\SET$, $\CAT$, $\SET^\Cat$ 
\\[1mm] \hline
\end{tabular}
\smallskip
\caption{Size of categories}
\label{tab:size}
\end{table}

\subsection*{Bicategories}

We shall assume familiarity with the fundamental notions and results of
2-dimensional category theory referring to~\cite{LackS:a2cc,JohnsonN:twodc}
for background.  This paper is written entirely in terms of the theory of
bicategories (aka weak 2-categories) to ensure that all our results are
invariant under the appropriate notion of equivalence. While some of the
structures that we deal with are sometimes presented in a stricter way, we
will not keep track of this to avoid proliferation of terminology, trusting
readers to be able to recognise situations of this kind.  
Hence, we will deal with bicategories, pseudofunctors (aka~homomorphisms),
pseudonatural transformations, and modifications. Accordingly, we will have
notions of biequivalence, biadjunction, bilimit and bicolimit of
pseudofunctors, and pseudomonad.  
We write $\ccatK \simeq \ccatL$ to indicate that two bicategories $\ccatK$ and
$\ccatL$ are biequivalent. 
By default, our bicategories will be large, \ie~have a large set of
equivalence classes of objects and large hom-categories.  Abusing notation, we
also write $\CAT$ for the bicategory of categories, and $\Cat$ for its full
sub-bicategory spanned by small categories.   
Small categories may be used to present categories by means of filtered
bicolimits.  

\begin{lem} \label{lem:ColimitPresentationForCAT}
Every category is a bicolimit in $\CAT$ of a $\lambda$-small $\univu$-filtered
diagram of $\univu$-small categories in $\Cat$.
\end{lem}

We recall some bicategorical notions and facts, see~\eg~\cite{StreetR:fibb}.
For a bicategory $\ccatK$ and objects $X, Y \in \ccatK$, we write 
$\hom\ccatK XY$ for the hom-category of maps and 2-cells from $X$ to~$Y$. For
maps $f \co X \to Y$ and $g \co Y \to Z$, we write $g \comp f \co X \to Z$ for
their horizontal composite.  
We say that a map $f \co X \to Y$ is an \emph{equivalence} if there is a
map $g \co Y \to X$ and isomorphisms $\eta \co \id_X \Rightarrow g \comp f$, 
$\varepsilon \co f \comp g \Rightarrow \id_Y$. When this happens, we say that
$X$ and $Y$ are \emph{equivalent} and write $X \simeq Y$. 
An \emph{adjunction} in $\ccatK$ consists of a pair of maps $f \co X \to Y$
and $g \co Y \to X$ together with 2-cells 
$\eta \co \id_X \Rightarrow g \comp f$, 
$\varepsilon \co f \comp g \Rightarrow \id_Y$ satisfying the usual triangular
laws.  
A pseudofunctor $G \co \ccatL \to \ccatK$ has a left biadjoint if and only if,
for each $X \in \ccatK$, there are an object $FX \in \ccatL$ and a map 
$\eta_X \co X \to GFX$ in $\ccatL$ universal in the sense of inducing an
equivalence
\[
\begin{tikzcd}[column sep = large]
\hom\ccatL{FX}Y 
\ar[r, "G_{FX,Y}"] & \hom\ccatK{GFX}{GY}
\ar[r, "\hom\ccatK{\eta_X}{GY}"] & \hom\ccatK X{GY} 
\end{tikzcd}
\]
for all $Y \in \ccatL$. 
Analogously to the 1-categorical setting, right biadjoints preserve bilimits
and left biadjoints preserve bicolimits. 

\subsection*{Coproducts and coexponentials}

We shall deal with bicategories with finite (bicategorical) products in which
certain objects are (bicategorically) exponentiable.  However, these will
arise as the opposite bicategories with (bicategorical) finite coproducts in
which certain objects are (bicategorically) coexponentiable.  For this reason,
we recall the latter notions.

\begin{defn} 
A bicategory $\ccatK$ has \myemph{binary coproducts} whenever the diagonal
functor $\ccatK \to \ccatK \times \ccatK$ has a left biadjoint.
\end{defn}
\noindent
Explicitly, a bicategory $\ccatK$ has binary coproducts if, for every pair of
objects $X_1, X_2 \in \ccatK$, there are an object $X_1 \copr X_2\in\ccatK$ 
and maps $\iota_1 \co X_1 \to X_1\copr X_2$ and 
$\iota_2 \co X_2 \to X_1 \copr X_2$ such that, for every $X \in \ccatK$,
composition with $\iota_1$ and $\iota_2$ induces an equivalence 
\[
\begin{tikzcd}[column sep = 3cm]
\hom\ccatK{X_1 \copr X_2} X 
\ar[r, "{\langle \hom\ccatK{\iota_1}X , \hom\ccatK{\iota_2}X \rangle}"] 
& \hom\ccatK{X_1}X \times \hom\ccatK{X_2}X 
\mathrlap{.}  
\end{tikzcd}
\]
In particular, for every pair of maps 
$f_1 \co X_1 \to X$ and $f_2 \co X_2 \to X$, we have a map 
$
X_1 \copr X_2 \to X$ 
and universal invertible 2-cells in the diagram below
\[
\begin{tikzcd}
X_1 \ar[dr, "f_1"'] \ar[r, "\iota_1"] & X_1 \copr X_2 \ar[d] & X_2 \ar[l, "\iota_2"']  \ar[dl, "f_2"]   \\
 & X \mathrlap{.} & 
\end{tikzcd}
\]

\begin{defn}  
A bicategory $\ccatK$ has an \myemph{initial object} whenever the unique
pseudofunctor $
\ccatK \to \mathsf 1$ has a left biadjoint.  
\end{defn}
\noindent
Explicitly, a bicategory $\ccatK$ has an initial object if there is an object
$\mathsf 0 \in \ccatK$ such that the unique functor 
$\hom\ccatK{\mathsf 0}X \to \mathrm{1}$ is an equivalence for every 
$X \in \ccatK$.  In particular, for every object $X \in \ccatK$, we have an
essentially unique map $
\mathsf 0 \to X$.

As usual, if a bicategory has binary coproducts and an initial object, then it
has all finite coproducts.

\begin{defn} 
Let $\ccatK$ be a bicategory with binary coproducts.  An object $X\in\ccatK$
is \myemph{coexponentiable} if the pseudofunctor 
$X\copr (-) \co \ccatK \to \ccatK$ has a left biadjoint.
\end{defn} 
\noindent
Explicitly, in a bicategory $\ccatK$, an object $X\in\ccatK$ is
coexponentiable if, for every object $Y \in \ccatK$, there are an 
object~${(X \coexp Y) \in \ccatK}$ and a map 
$
Y \to X \copr (X \coexp Y)$ 
inducing an equivalence
\[
\hom\ccatK{X \coexp Y}Z  
\simeq
\hom\ccatK Y{X \copr Z}
\]
for all $Z\in\ccatK$.

\subsection*{Dualizability and nuclearity}

Let $(\ccatK, \otimes, \one)$ be a monoidal bicategory~\cite{JohnsonN:twodc}.  
An object $X\in \ccatK$ is \myemph{dualizable} if there is an object
$X\dual \in \ccatK$, called the \myemph{dual} of $X$, with two maps 
$\eta \co \one \to X \otimes X\dual$, called the \myemph{coevaluation}, and 
$\varepsilon \co X\dual \otimes X  \to \one$, called the \myemph{evaluation},
such that the composite maps
\begin{equation}
\label{eq:dual}
	\begin{tikzcd}
		X \simeq \one \otimes X\ar[r, " \eta \otimes\id"]
		& (X \otimes X\dual )\otimes X
    \simeq 
		 X \otimes (X\dual \otimes X)
		\ar[r, "\id\otimes\varepsilon"]
		& X \otimes \one \simeq X 
    \mathrlap{,}
		\\
		X\dual \simeq X\dual \otimes \one
		\ar[r, "\id\otimes\eta"] 
		& X\dual \otimes (X \otimes X\dual)
    \simeq
		 (X\dual \otimes X) \otimes X\dual
		\ar[r, "\varepsilon\otimes\id"]
		& \one\otimes X\dual \simeq X\dual
	\end{tikzcd}
\end{equation}
are isomorphic to the identities of $X$ and $X\dual$, respectively.  
We note for the record that every such structure, called a \emph{dual pair},
can actually be turned into a structure equipped with invertible 2-cells
satisfying swallowtail coherence conditions, see~\cite[Definition~3.11 and
Theorem~3.14]{Pstrgowski:OnDualizableObjects}.

We henceforth rely on coherence theorems for symmetric monoidal
bicategories~\cite{GordonR:coht,GurskiN:cohtdc} eliding coherent pseudonatural
maps to simplify both notation and the development, for instance, as in the
following.
Let $X, X\dual\in \ccatK$ be a dual pair.  Then, for all objects $Y,Z\in
\ccatK$,
we have pseudonatural equivalences
\begin{align*}
\hom\ccatK{Y\otimes X}{Z}
\xrightarrow{-\otimes X\dual}
\hom\ccatK{Y\otimes X\otimes X\dual}{Z\otimes X\dual}
\xrightarrow{\hom\ccatK{Y\otimes\eta}{Z\otimes X\dual}}
\hom\ccatK{Y}{Z\otimes X\dual}
\mathrlap,
\\[2mm]
\hom\ccatK{Y\otimes X}{Z}
\xleftarrow{\hom\ccatK{Y\otimes X}{Z\otimes \varepsilon}}
\hom\ccatK{Y\otimes X}{Z\otimes X\dual\otimes X}
\xleftarrow{-\otimes X}
\hom\ccatK{Y}{Z\otimes X\dual}
\mathrlap.
\end{align*}
Thus, $(-)\otimes X$ has right biadjoint $(-)\otimes X\dual$.  
Moreover, if $\ccatK$ is symmetric monoidal, then the pair $X\dual,X\in\ccatK$
is also dual and we additionally have that $(-)\otimes X\dual$ has right
biadjoint
$(-)\otimes X$.  
We therefore have the following.

\begin{prop} \label{prop:DualInSMCBiCs}
For every dual pair $X,X\dual$ in a symmetric monoidal closed bicategory
with structure $\big((-)\otimes(=),\one,\map-=\big)$, 
\[
  X\dual\otimes(-)
  \ \dashv \
  (-)\otimes X
  \ \dashv \
  X\dual\otimes(-)
  \ \simeq \
  \map X-
  \mathrlap.
\]
Moreover, for every object $Y$, we have a canonical equivalence
\[
  \map X\one \otimes Y 
  \xrightarrow{\simeq}
  X\dual \otimes Y
  \xrightarrow{\simeq}
  \map X Y
\]
isomorphic to the canonical strength map 
\[
\stmap_{X,Y} \co \map X\one \otimes Y \to \map X Y
\]
given by the adjoint transpose of the composite 
$\map X\one \otimes Y \otimes X
 \simeq
 \map X\one \otimes X \otimes Y
 \to 
 \one\otimes Y 
 \simeq 
 Y$
arising from the evaluation map.
\end{prop}
\noindent
It follows that every dual object is nuclear in the sense of the following
definition that directly generalizes the established corresponding categorical
notion~\cite{Rowe}.

\begin{defn}
An object $X$ in a symmetric monoidal closed bicategory is called
\myemph{nuclear} whenever the canonical strength map 
$\stmap_{X,X} \co \map X\one \otimes X \to \map X X$ is an equivalence.
\end{defn}

The folklore equivalence between the categorical notions of a dual and a
nuclear object (see~\eg~\cite{deLacroixSantocanale}) generalizes to the
bicategorical setting.

\begin{thm} \label{thm:DualizableEqualNuclear}
In a symmetric monoidal closed bicategory, an object is dualizable if and only
if it is nuclear.
\end{thm}
\begin{proof}
$(\Rightarrow)$ Follows from \cref{prop:DualInSMCBiCs}.

$(\Leftarrow)$
For an object $X$, let $\one \to \map XX$ be the adjoint transpose of the
canonical equivalence $\one\otimes X\simeq X$.  For $X$ nuclear, we have a
dual pair $X, \map X\one$ with coevaluation and evaluation maps given by
\[
  \one \to \map XX 
  \simeq X\otimes \map X\one
  \enspace,\quad
  \map X\one \otimes X \to \one
  \mathrlap{.}
  \qedhere
\]
\end{proof}

In a bicategory, a map $f \co X\to Y$ is said to be a \myemph{pseudo-section}
(respectively, \myemph{pseudo-retraction}) if there is a map 
$g \co Y \to X$ such that the endomap $gf$ on $X$ (respectively, $fg$ on $Y$) 
is isomorphic to the identity on $X$ (respectively, $Y$).  In this case, we
say that $X$ is a \myemph{deformation retract} of $Y$.

\begin{lem}%
\label{lem:Rectractdualizable}
In a symmetric monoidal closed bicategory, any deformation retract of a
nuclear object is nuclear.
\end{lem}
\begin{proof}
Because, for a deformation retract $X$ of an nuclear object $Y$, the canonical
strength map $\stmap_{X,X}$ is a deformation retract of the equivalence
canonical strength map $\stmap_{Y,Y}$.
\end{proof}

\section{Cocomplete categories and dualizability}
\label{sec:dual-coc}
\label{sec:coc}

\subsection*{Cocomplete categories} 
\label{subsec:CocompleteCategories}

We say that a category is \emph{cocomplete} if it has small colimits
(\ie~colimits indexed by small categories).  Specifically, recalling the size
conventions set in \cref{subsec:SizeConventions}, a cocomplete category is a
$\lambda$-small category with colimits of $\kappa$-small diagrams.  Beware
that we are not imposing that categories are locally small.	

We write $\COC$ for the large bicategory of cocomplete categories,
cocontinuous functors, and natural transformations. 

\begin{ex}\leavevmode
\label{ex:coc}
\begin{enumerate}
\item\label{ex:coc:zero} 
  The terminal category is in $\COC$, where it is a zero object (\ie~both
  initial and terminal) and therein denoted $\zeroobj$. 

\item 
  The categories $\Set$ and $\Cat$ are in $\COC$, as so are the functor
  categories $\Set^A$ and $\Cat^A$ for every category $A$.
\end{enumerate}
\end{ex}

We consider closure under colimits in the vein of
\cite{Lindner,KellyGM:bascec,AlbertKelly}.

\begin{defn} \label{defn:coc:spsh} 
For a category $\catA$, define $\spshA$ as the smallest large full subcategory
of the large functor category $\SET^{A^\op}$ containing the representable
functors $A(-,a) \co A^\op \to \SET$, for all $a\in A$, and closed under small
colimits.  
\end{defn}

For every category $\catA$, under our assumption $\kappa>\aleph_0$, it can be
shown that $\spshA$ is equivalent to the category spanned by small colimits of
representable functors in $\SET^{A^\op}$, and hence that it is in $\COC$.  The
category $\spshA$ is henceforth referred to as the \emph{presheaf category}
over the category $A$.  In particular, the presheaf category $\spshA$ over a
small category $A$ is the functor category $\Set^{A^\op}$.

A cocomplete category is said to be \emph{presentable} 
if it can be presented as a localisation in $\COC$ of a presheaf category over
a small category inverting a small set of maps.  
We denote by $\Pres$ the full sub-bicategory of $\COC$ spanned by the
presentable categories. 

\medskip
The \emph{presheaf construction} $\spsh$ (\cref{defn:coc:spsh}) together with
the Yoneda embeddings 
\begin{align*}
\yon_\catA \co \catA \embto \spsh(\catA)
\qquad
(A\in\CAT)\mathrlap{,}
\end{align*}
mapping objects to representable functors, induce the equivalences
\begin{align*}
  (-)\circ\yon_\catA 
  \co  \COC(\spsh\catA,\catC) \xrightarrow{\enspace\simeq\enspace} \CAT(A,C)
  \qquad
  (A\in\CAT, C\in\COC)
\end{align*}
and provides a left biadjoint to the inclusion pseudofunctor 
$\sfUin \co \COC \to \CAT$;
see~\cite{KockA:monwua,ZoberleinV:doc2c,KellyGM:bascec}, with modern accounts
in~\cite{LackS:monccc,GarnerR:lexc}.

\begin{prop} \label{prop:coc-monadic} 
There is a pseudomonadic biadjunction
\begin{equation} \label{display:CAT-COC-adjunction}
\begin{tikzcd}
\CAT
\ar[r, shift left = 2, "\spsh"] 
\ar[r, phantom, description, "\scriptstyle{\vdash}" rotate=90]
&
\ar[l, shift left = 2 , "\sfUin"] 
\COC
\end{tikzcd}
\end{equation}
with $\spsh$ mapping $\Cat\fullsubcat\CAT$ to $\Pres\fullsubcat\COC$.
Moreover, the induced monad $\sfUin\spsh$ preserves bicolimits
of $\lambda$-small $\kappa$-filtered diagrams.
\end{prop} 

\begin{lem} \label{lem:ColimitPresentationForCOC}
Every cocomplete category is a bicolimit of a $\lambda$-small diagram of
presheaf categories over $\kappa$-small categories.  
\end{lem}
\begin{proof}
Let $C$ be a cocomplete category. 
\begin{enumerate}
\item
By \cref{lem:ColimitPresentationForCAT}, $C$ is a bicolimit in $\CAT$ of a
$\lambda$-small $\kappa$-filtered diagram of $\kappa$-small categories 
$A_i\ (i\in I)$ in $\Cat$. 

\item 
By the pseudomonadicity~\cite[Theorem~3.5]{CreurerI:bectpm} of the
pseudoadjunction~\cref{display:CAT-COC-adjunction}, $C$ is a bicolimit in
$\CAT$ of a codescent diagram (omitting the iso 2-cells) 
\[\begin{tikzcd}
\spsh^3 C  
\ar[r, shift left = -3] 
\ar[r, shift left = 3] 
\ar[r, shift left = 0] 
& 
\spsh^2 C  
\ar[r, shift left = -3] 
\ar[r, shift left = 3] 
& 
\ar[l]
\spsh C  
\mathrlap.
\end{tikzcd}\]
\end{enumerate}

Thus, as the pseudomonad of the
pseudoadjunction~\cref{display:CAT-COC-adjunction} is $\kappa$-finitary, the
diagram of $I$-indexed codescent diagrams (again omitting the iso 2-cells) in
$\CAT$ 
\[\begin{tikzcd}
\spsh^3 A_i
\ar[r, shift left = -3] 
\ar[r, shift left = 3] 
\ar[r, shift left = 0] 
& 
\spsh^2 A _i
\ar[r, shift left = -3] 
\ar[r, shift left = 3] 
& 
\ar[l]
\spsh A_i
&
(i\in I)
\end{tikzcd}\]
has bicolimit $C$ in $\CAT$.
\end{proof}

\medskip
Let $I$ be a small set.  Recall that the functors of sum $\coprod_I$ and of
product $\prod_I$ indexed by $I$ are respectively left and right biadjoints to
the constant $I$-family pseudofunctor $\COC\to \COC^I$.  Recall also from
\cref{ex:coc}\,\ref{ex:coc:zero}, that the terminal category is a zero object
$\zeroobj$ in $\COC$.
Given a family $D \co I \to \COC$, let $\delta \co I \times I \to \COC$ be
given by $\delta(i,j) \defeq (D_i\to \zeroobj \to D_j)$ if 
$i\varnot= j$ and $\delta(i,i) \defeq \Id_{D_i}$ otherwise.  
This \emph{Kronecker delta} universally induces a cocontinuous functor
$\coprod_I D \to \prod_I D$, and we say that the $I$-sum and the $I$-product
coincide whenever it is an equivalence; in which case, this \emph{sum-product}
is denoted $\oplus_I\,D$.

\begin{lem}
\label{lem:coc-additive}
Small sums and small products in $\COC$ coincide. 
\end{lem}

The category of cocontinuous functors from a cocomplete category $\cocC$ to
another one $\cocD$, and natural transformations between them, is cocomplete
and denoted $\inthom\COC\cocC\cocD$.  This \myemph{internal hom} has an
associated tensor product~\cite{KellyGM:bascec} that we now recall.
For $\cocE$ a cocomplete category, the 
category 
$\hom\COC \cocC{\inthom\COC \cocD\cocE}$ is equivalently the category of
separately-cocontinuous functors $\cocC\times \cocD\to \cocE$.  The tensor
product $\cocC\otimes \cocD$ is defined as the codomain of a
separately-cocontinuous functor 
${\otimes\co\cocC\times\cocD\to \cocC\otimes \cocD}$ representing the
pseudofunctor $\Hom\COC \cocC{\inthom\COC \cocD{-}}: \COC\to\CAT$.
Explicitly, for every separately-cocontinuous functor 
$F:\cocC \times \cocD\to \cocE$, there is an essentially unique factorisation 
\[
\begin{tikzcd}
\cocC \!\times\! \cocD \ar[r,"\otimes"] \ar[dr, "F"'] & 
\cocC \otimes \cocD \ar[d, "F^\dag"] \\
& \cocE 
\end{tikzcd}
\]
where $F^\dag$ is a cocontinuous functor.

\begin{prop} \label{prop:TensorOfPresheafCats}
For categories $\catA$ and $\catB$, 
\[
  \spsh(\catA)\otimes\spsh(\catB) \simeq \spsh(\catA \times \catB)
\mathrlap{.}
\]
\end{prop}
\begin{proof}
It follows from the equivalences:
\begin{align*}
\hom\COC{\spsh(\catA \times \catB)}{-}
&\simeq \hom\CAT{\catA \times \catB}{-}
\\
&\simeq \Hom\CAT{\catA}{\hom\CAT \catB-}
\\
&\simeq \Hom\CAT{\catA}{\hom\COC{\spsh(\catB)}-}
\\
&\simeq \Hom\COC{\spsh(\catA)}{\inthom\COC{\spsh(\catB)}-}
\end{align*}
for all categories $\catA$ and $\catB$.
\end{proof}

In general, for cocomplete categories $\cocC$ and $\cocD$, the cocomplete
category $\cocC \otimes \cocD$ can be constructed as the localisation in
$\COC$ of $\spsh(\cocC \!\times\! \cocD)$ inverting the small set of canonical
maps $\big[(\iota_j,\id)\big]_j \co \colim_j(X_j,Y)\to (\colim_j X_j,Y)$ and
$\big[(\id,\iota_k)\big]_k \co \colim_k(X,Y_k)\to (X, \colim_k Y_k)$, for all
small diagrams $X:J\to \cocC$ and $Y:K\to \cocD$.  This explicit description
shows that if $\cocC$ and $\cocD$ are presentable, then so is
$\cocC\otimes\cocD$.
The unit of the tensor product is the category $\Set$ since 
$\hom\COC\Set \cocC \simeq \cocC$, for every cocomplete category $\cocC$.  
It may be shown that this defines a symmetric monoidal closed structure on
$\COC$, see~\cite{KellyGM:bascec,HylandM:psecmpc}.

\begin{prop}
\label{prop:COC-SMC}
The structure $(\otimes,\Set,\inthom\COC-=)$ makes $\COC$ into a symmetric
monoidal closed bicategory.  Moreover, $\Pres\fullsubcat\COC$ is a symmetric
monoidal closed sub-bicategory.
\end{prop}

\subsection*{Dualizable cocomplete categories}
\label{sec:dual}
\label{sec:tensor-COC}

We now turn to the characterization of dualizable objects in the symmetric
monoidal closed bicategory $\COC$. See~\cite{HarpazY:mathoverflow} for a related discussion
on dualizable objects in $\Pres$. 

\begin{thm}%
\label{thm:dualizable-coc}
The dualizable objects of $\COC$ are the cocontinuous deformation retracts of
presheaf categories over small categories.
\end{thm}
\begin{proof}
We first note that, for all small categories $A$, the presheaf category 
$\PSh A$ is dualizable, with dual $\PSh {A^\op}$.  Indeed, this equivalently
corresponds to the compact closed structure of the bicategory of small
categories, profunctors, and natural transformations,
see~\eg~\cite{DayB:monbha}.  Thus, by \cref{thm:DualizableEqualNuclear} and
\cref{lem:Rectractdualizable}, cocontinuous deformation retracts of presheaf
categories over small categories are dualizable.
 
Reciprocally, let $C$ be a dualizable cocomplete category with dual $C\dual$.
By \cref{prop:DualInSMCBiCs}, $C\dual \simeq \inthom\COC C\Set$ and 
$C\dual\otimes C \simeq \inthom\COC CC$.
Using the construction of the tensor product of cocomplete categories recalled
in \cref{subsec:CocompleteCategories}, we have a localisation
\[
\begin{tikzcd}
	\spsh{( C\dual \!\times\! C )} \ar[r] &
	C\dual \otimes C \simeq \inthom\COC CC \mathrlap{.}
\end{tikzcd}
\]
Since localisations are essentially surjective functors, there is a small
category $J$ and a diagram 
$(\gamma_\bullet, c_\bullet):J\to C\dual \times C$ such that 
$\colim_j (\gamma_j\otimes c_j)$ in $C\dual \otimes C$ is equivalently $\Id_C$
in $\inthom\COC C C$.  This means that, for every object $c$ in $C$, we have a
decomposition
\[\textstyle
\colim_j ( \gamma_j(c) \tensor c_j ) \cong c	
\mathrlap{,}
\]
where $\gamma_j(c)$ is in $\Set$ and $\gamma_j(c)\tensor c_j$ is the tensor of
$C$ over $\Set$.  We are now going to use a different decomposition of
$\Id_C$ in terms of a coend over $J$.  By the Density Formula, we have 
$c_j \cong \int^{k\in J} \hom J kj \tensor c_k$ and thus 
\begin{equation}
\label{eq:trick-coend}
c 
\ \cong \
\colim_j
  \left(
    \gamma_j(c)\tensor \left(\int^{k\in J} \hom J kj \tensor c_k \right)
  \right)
\ \cong \
\int^{k\in J}
  \big(
  \underbrace{\colim_j \big( \gamma_j(c)\times \hom J kj \big)}_{\beta^k_c}
  \big)
  \tensor 
  c_k 
\mathrlap{.}
\end{equation}
For $k$ in $J$ and $c$ in $C$, the small sets 
$\beta^k_c = \colim_j \big(\gamma_j(c) \times  \hom J kj\big)$ assemble into a
cocontinuous functor $S \co C \to \PSh J$.  Then, for $R \co \PSh J \to C$ the
cocontinuous extension of the functor $c_\bullet \co J \longrightarrow C$, the
formula~\cref{eq:trick-coend} establishes that $\Id_C\cong RS$.  This shows
that $C$ is a cocontinuous deformation retract of a presheaf category over a
small category.
\end{proof}

\section{Symmetric 2-rigs and coexponentiability}
\label{sec:sym2rig}

We study the fundamental mathematical structure of the paper: symmetric
2-rigs.  These are categories with \emph{multiplicative} symmetric monoidal
structure distributing over \emph{additive} small colimit structure.  It is
illuminating to look at these both in isolation and in combination.  The
additive structure was the subject of \cref{subsec:CocompleteCategories}.
Herein, the multiplicative structure is considered in \cref{subsec:smoncat}
and, thereafter, their combination is analyzed in \cref{sec:2-rigs}.

\subsection*{Symmetric monoidal categories} 
\label{subsec:smoncat}

We write $\SMONCAT$ for the large bicategory of symmetric monoidal categories,
symmetric strong monoidal functors, and monoidal transformations.  

A symmetric monoidal category can be equivalently defined as a pseudo
symmetric monoid in $(\CAT,\times,\terminalcat)$ and the forgetful
pseudofunctor $\sfU \co \SMONCAT \to \CAT$ can be biequivalently regarded as
the associated forgetful pseudofunctor $\PsSMon_\times(\CAT) \to \CAT$,
mapping a pseudo symmetric monoid to its underlying category.  The following
result follows from~\cite{BlackwellR:twodmt}. 

\begin{prop} \label{prop:smoncat-monadic} 
There is a pseudomonadic biadjunction
\begin{equation} \label{eqn:CAT-SMCAT-adjunction}
\begin{tikzcd}
\CAT \ar[r, shift left = 2, "\freesmc"] 
\ar[r, phantom, description, "\scriptstyle{\vdash}" rotate=90] & 
\ar[l, shift left = 2, "\sfU"] \SMONCAT \mathrlap{.} 
\end{tikzcd}
\end{equation}
Moreover, the induced monad $\sfU\freesmc$ preserves filtered colimits.
\end{prop}

Free symmetric monoidal categories have a simple description.  Let 
$\mathfrak S = \coprod_{n\in\mathbb N} \mathfrak S_n$  be the groupoid of
finite symmetric groups.  Then, every category $\catA$ induces a
$\mathfrak S$-indexed category $\catA^{(-)}:\mathfrak S \to \CAT$ sending $n$
to $\catA^{n}$ and the category $\freesmc\catA$ arises as its Grothendieck
construction.  For instance, $\freesmc(\terminalcat) \cong \mathfrak S$.  In
abstract terms, 
\[\textstyle
\freesmc\catA
\ =\ 
\bicolim_{\sigma\in\mathfrak S}\ \catA^{\sigma}
\ =\ 
\coprod_{n\in\Nat}\ \bicolim_{\sigma\in\mathfrak S_n}\catA^{\sigma}
\mathrlap{,} 
\]
where each summand is the bicategorical quotient of the $n$-fold product by
the action of $\mathfrak S_n$.

As discussed in~\cite{HylandM:psecmpc}, see also~\cite{BourkeJ:skes2c}, the
pseudomonad $\sfU\freesmc \co \CAT \to \CAT$, whose algebras are symmetric
monoidal categories, is pseudocommutative; and, thereby, $\SMONCAT$ comes
equipped with a symmetric monoidal closed structure 
$\big(\otimes,\mathfrak S,\inthom\SMCAT-=\big)$.  
For permutative~(\viz~symmetric strict monoidal) categories, this monoidal
structure has been defined and studied in detail in~\cite{GurskiN:symmbp}.  

\begin{prop} \leavevmode
\begin{enumerate}
\item
Finite coproducts and finite products in $\SMONCAT$ coincide. 

\item
For categories $A$ and $B$,
\[
  \freesmc(\initialcat) \simeq \terminalcat
  \mathrlap,\quad
  \freesmc(A+B) \simeq \freesmc(A)\times\freesmc(B)
  \mathrlap,\quad
  \freesmc(A)\otimes \freesmc(B) \simeq \freesmc(A \times B)
  \mathrlap{.} 
\]
\end{enumerate}
\end{prop}

\subsection*{Symmetric 2-rigs} 
\label{sec:2-rigs}

\begin{defn} 
The bicategory $\RIG$ ($\Rig$) has:
\begin{itemize}
\item 
  objects given by (presentable) \emph{symmetric 2-rigs},~\ie~cocomplete
  symmetric monoidal (presentable) categories in which the tensor product is
  separately cocontinuous,

\item 
  morphisms given by symmetric strong monoidal cocontinuous functors, and

\item 
  2-cells given by symmetric monoidal natural transformations.  
\end{itemize}
\end{defn} 
\noindent
All 2-rigs considered in the paper are symmetric.  
Therefore, we shall henceforth simply speak of \mbox{2-rigs} rather than of
symmetric 2-rigs. 

Write $\mylift\sfU\co\RIG\to\COC$ and $\mylift\sfUin\co\RIG\to\SMCAT$ for the
forgetful pseudofunctors.

The category of presheaves $\spshM$ over a symmetric monoidal category
$\moncatM$ is equipped with a symmetric monoidal structure, known as
\emph{convolution}~\cite{DayB:clocf}, 
for which the Yoneda embedding 
$\yon_\moncatM \co \moncatM \embto \spsh\moncatM$ is symmetric strong monoidal
and has the following universal property~\cite{ImG:unipcm}:
\begin{align*}
  (-)\circ\yon_M
  \co  
  \RIG(\spsh M,R) \xrightarrow{\enspace\simeq\enspace} 
  \SMONCAT(M,R)
  \qquad
  (R\in\RIG)
  \mathrlap{.}
\end{align*}
Thus, there is a biadjunction 
\begin{equation} \label{eqn:SMCAT-RIG-adjunction}
\begin{tikzcd}
\SMCAT \ar[r, shift left = 2, "\mylift\spsh"] 
\ar[r, phantom, description, "\scriptstyle{\vdash}" rotate=90] & 
\ar[l, shift left = 2, "\mylift\sfUin"] \RIG 
\end{tikzcd}
\end{equation}
whose induced pseudomonad $\mylift\sfUin\mylift\spsh$ on $\SMCAT$ is a lifting
as follows
\begin{equation} \label{diag:pseudolifting}
\begin{tikzcd}
    \CAT \ar[d,"\sfUin\spsh"'] 
    & \SMCAT \ar[d,"\mylift\sfUin\mylift\spsh"] 
    \ar[l,"\sfU"'] 
    \\
    \CAT & \SMCAT
    \ar[l,"\sfU"] 
\end{tikzcd}
\end{equation}
of the pseudomonad $\sfUin\spsh$ on $\CAT$ induced by the
biadjunction~\cref{display:CAT-COC-adjunction} whose algebras are,
equivalently, cocomplete categories.

A $\mylift\sfUin\mylift\spsh$-pseudoalgebra 
$a \co (\spsh M,\mylift\otimes,\yon\one) \to (M,\otimes,\one)$ has an
underlying $\sfUin\spsh$-pseudoalgebra $a \co \spsh M \to M$, and the category
$M$ is therefore cocomplete.  In particular, the diagram
\[\begin{tikzcd}
  M^I \ar[r,"\colim"] \ar[d,"(\yon_M)^I"'] & M 
  \\
  (\spsh M)^I \ar[r,"\colim"'] & \spsh M \ar[u,"a"']
\end{tikzcd}\]
commutes up to canonical isomorphism for every small category $I$.  Moreover,
for all $X\in M^I$ and $Y\in M$, the canonical isomorphisms
\begin{align*} \textstyle 
(\colim_i X_i) \otimes Y
&\textstyle 
\cong a(\colim_i\, \yon X_i) \otimes a(\yon Y)
\\ 
&\textstyle 
\cong a\big((\colim_i\, \yon X_i) \ \mylift\otimes\ \yon Y)\big)
\\ 
& \textstyle 
\cong a\Big(\colim_i \big(\yon(X_i) \ \mylift\otimes\ \yon(Y)\big)\Big)
\\ 
& \textstyle 
\cong a\big(\colim_i\, \yon(X_i\otimes Y)\big)
\\ 
& \textstyle 
\cong \colim_i\, (X_i\otimes Y)
\end{align*}
establish that the symmetric monoidal category $M$ has a separately
cocontinuous tensor product.  It follows that every
$\mylift\sfUin\mylift\spsh$-pseudoalgebra is a 2-rig.  In fact, this may be
strengthen as follows.

\begin{prop}
The biadjunction~\cref{eqn:SMCAT-RIG-adjunction} is pseudomonadic.
\end{prop}

Recalling that the biadjunction~\cref{eqn:CAT-SMCAT-adjunction}, inducing the
monad $\sfU\freesmc$ on $\CAT$, is pseudomonadic, by general lifting
results~\cite{MarmolejoF:dislp,ChengE:psedl,TanakaM:psedl,FioreM:relpkbs}, 
the lifting~\cref{diag:pseudolifting} corresponds to a pseudodistributive law 
$\sfU\freesmc\ \sfUin\spsh \to   \sfUin\spsh\ \sfU\freesmc$, for which the
reader may consult~\cite{FioreM:carcbg}.
This establishes the following.
\begin{prop} \label{prop:pseudomonadicitysquare}
There is a commutative square of pseudomonadic forgetful pseudofunctors as
follows 
\begin{equation}\label{diag:monadicitysquare}
\begin{tikzcd}[column sep = large] 
\CAT
\ar[rr, "\freesmc", shift left = 3, dashed]
\ar[rr, phantom, description, "\scriptstyle{\vdash}" rotate=90, shift left = 1] 
\ar[dd, phantom, description, "\scriptstyle{\dashv}"] 
\ar[dd, "\spsh"', shift right = 2, dashed]
\ar[rrdd, "\mylift\spsh\freesmc"', shift right = 2, dashed]
\ar[rrdd, phantom, description, "\scriptstyle{\dashv}",shift right=0.25] 
&&
\SMONCAT
\ar[ll, shift left = 1,"\sfU"]
\ar[dd, "\mylift\spsh"', shift right = 2, dashed]
\ar[dd, phantom, description, "\scriptstyle{\dashv}"] 
\\
\\
\COC 
\ar[uu, shift right = 2,"\sfUin"']
&& \RIG 
\ar[lluu, shift right = 2,"\sfU\mylift\sfUin"', shorten > = 3 ex]
\ar[uu, shift right = 2,"\mylift\sfUin"']
\ar[ll,shift left = 2,"\mylift\sfU"]
\end{tikzcd}
\end{equation}
with the pseudomonad $\sfU\mylift\sfUin\mylift\spsh\freesmc$ equivalently
described by the pseudomonad $\sfUin\spsh\sfU\freesmc$.  
\end{prop}

Furthermore, the above picture completes as follows.
\begin{prop} 
\label{thm:sym-alg-coc}
\label{sec:2rig-pseudomonoids}
There is a pseudomonadic biadjunction
\begin{equation}
\label{diag:2rig-pseudomonoids}
\begin{tikzcd}
\COC \ar[r, shift left = 2, "\mylift\Sym"] 
\ar[r, phantom, description, "\scriptstyle{\vdash}" rotate=90] & \RIG
\ar[l, shift left = 2, "\mylift\sfU"] 
\mathrlap{.} 
\end{tikzcd}
\end{equation}
Moreover, this restricts to a pseudomonadic biadjunction 
$\mylift\Sym \co \Pres \rightleftarrows \Rig \co \mylift\sfU$.
\end{prop}

Indeed, recalling the symmetric monoidal structure on cocomplete categories
from \cref{sec:tensor-COC}, a 2-rig can be biequivalently defined as a
symmetric pseudomonoid in $(\COC,\otimes,\Set)$ and the forgetful
pseudofunctor $\RIG \to \COC$ can be biequivalently regarded as the associated
forgetful pseudofunctor $\PsSMon_\otimes(\COC) \to \COC$, mapping a symmetric
pseudomonoid to its underlying cocomplete category.  
Since the tensor product in $\COC$ is separately cocontinuous, the free
symmetric pseudomonoid on a cocomplete category $C$ has a standard direct
description, obtained as the bicategorical quotient by the action of the
symmetric groups on finite tensor powers; specifically, 
\[\textstyle
\mylift\Sym\cocC
\ =\ 
\oplus_{n\in\mathbb N}\ \mylift\Sym_n \cocC 
\mathrlap{,}\quad\mbox{where}\enspace
\mylift\Sym_n\cocC = \bicolim_{\sigma\in\mathfrak S_n}\cocC^{\otimes\sigma}
\mathrlap.
\]
Note that when~$\cocC$ is a presentable category, by~\cref{prop:COC-SMC}, so
are the~$\cocC^{\otimes n}$, the $\mylift\Sym_n\cocC$, and $\mylift\Sym\cocC$.

On the other hand, the pseudomonadicity of $\RIG$ over
$\COC$~\cref{diag:2rig-pseudomonoids} is established using Beck's theorem for
pseudomonads~\cite[Theorem~3.6]{CreurerI:bectpm} relying on the
pseudomonadicity of $\COC$ and $\RIG$ over $\CAT$~\cref{diag:monadicitysquare}
and the pseudoconservativity of $\mylift\sfU$.

\medskip
In analogy with the situation from commutative algebra~\cite{EisenbudCommAlg}
and topos theory~\cite{BungeM:symt}, we shall refer to $\mylift\Sym$ as the
\myemph{symmetric algebra} pseudofunctor.  Correspondingly, a 2-rig in its
essential image will be called a \myemph{symmetric-algebra 2-rig}.  
We record the following consequence of the
pseudomonadicity~\cite[Theorem~3.5]{CreurerI:bectpm}
of~\cref{diag:2rig-pseudomonoids}.

\begin{lem}
\label{lem:sym-colim}
Every symmetric 2-rig is a bicolimit of a codescent diagram of
symmetric-algebra 2-rigs.  
\end{lem}

\subsection*{Convolution, operadic, and free symmetric 2-rigs} 

Diagram~\cref{diag:monadicitysquare} naturally lead us to consider two full
sub-bicategories of $\RIG$: the bicategory $\CRIG$ of \emph{convolution}
2-rigs and its full sub-bicategory $\FRIG$ of \emph{free} 2-rigs.  These are
respectively defined as the essential images of the pseudofunctor 
\[
  \mylift\spsh : \SMCAT \to \RIG
\]
and of its restriction
\[
  \mylift\spsh\freesmc \simeq \mylift\Sym\spsh : \CAT \to \RIG
  \mathrlap.
\]

We let $\CRig = \CRIG\cap\Rig$ and $\FRig = \FRIG\cap\Rig$; that is, $\FRig$
(resp.~$\CRig$) is the full sub-bicategory of $\Rig$ spanned by free
(resp.~convolution) 2-rigs on a (resp.~symmetric monoidal) \emph{small}
category.

We further consider an intermediate class of 2-rigs which we call 
\emph{operadic}.  These arise from a biadjuction 
\[\begin{tikzcd}
\OPD 
\ar[r, shift left = 2, "\Env"] 
\ar[r, phantom, description, "\scriptstyle{\vdash}" rotate=90] & 
\ar[l, shift left = 2, "\End"] \SMCAT 
\end{tikzcd}\]
between the bicategory of coloured operads $\OPD$ and the bicategory of
symmetric monoidal categories $\SMCAT$, details for which may be found
in~\cite[\S{5.2}]{BataninM:regpsf} and~\cite[\S{4}]{AnelM:ope2r}.  
Thereby, we let $\OpdRIG$ be the full sub-bicategory of $\CRIG$ given by the
essential image of the pseudofunctor 
\[
  \mylift\spsh \: \Env: \OPD \to \CRIG
  \mathrlap.
\]
We also let $\OpdRig = \OpdRIG\cap\Rig$.

The following diagram summarizes 
the bicategories defined above:
\[\begin{tikzcd}
\FRig \ar[r,hook] \ar[d,hook]
& \OpdRig \ar[r,hook] \ar[d,hook]
& \CRig \ar[r,hook] \ar[d,hook]
& \Rig \ar[d,hook] 
&
\\
\FRIG \ar[r,hook] 
& \OpdRIG \ar[r,hook]
& \CRIG \ar[r,hook] 
& \RIG 
& \hspace{-15mm}.
\end{tikzcd}\]

\subsection*{Finite coproducts of symmetric 2-rigs} 
\label{Subsection:CoproductsOf2Rigs} 

\begin{prop} \label{prop:RIGcocartesian}
The bicategory $\RIG$ admits finite coproducts, given by finite tensor
products of cocomplete categories. 
\end{prop}
\begin{proof}   
This is an instance of a general fact concerning bicategories of symmetric
pseudomonoids in symmetric monoidal
bicategories, see~\cite[Theorem~5.2]{SchappiD:indacq}.

Specifically, the initial 2-rig is the cartesian monoidal category $\Set$,
while the coproduct $R+S$ of 2-rigs $R$ and $S$, has underlying cocomplete
category $\mylift\sfU(R)\otimes\mylift\sfU(S)$ and underlying symmetric
monoidal category with tensor product and unit given pointwise:
\[
\mylift\sfU(R)\otimes\mylift\sfU(S)\otimes\mylift\sfU(R)\otimes\mylift\sfU(S)
\simeq
\mylift\sfU(R)\otimes\mylift\sfU(R)\otimes\mylift\sfU(S)\otimes\mylift\sfU(S)
\longrightarrow
\mylift\sfU(R)\otimes\mylift\sfU(S)
\longleftarrow
\Set\otimes\Set
\simeq
\Set
\mathrlap.
\qedhere
\]
\end{proof} 

\begin{prop} \label{prop:coprod} \leavevmode
The bicategories 
\begin{enumerate}
\item \label{prop:conv-coprod}
  $\CRIG$ and $\CRig$ of convolution 2-rigs, 

\item \label{prop:opd-coprod} 
  $\OpdRIG$ and $\OpdRig$ of operadic 2-rigs, and 
  
\item \label{prop:free-coprod}
  $\FRIG$ and $\FRig$ of free 2-rigs 
\end{enumerate}
are cocartesian.
\end{prop} 
\begin{proof} 
Because the bicategories \ref{prop:conv-coprod}~$\SMCAT$, 
\ref{prop:opd-coprod}~$\OPD$, \ref{prop:free-coprod}~$\CAT$ and the
pseudofunctors \ref{prop:conv-coprod}~$\mylift\spsh \co {\SMCAT\to\RIG}$,
\ref{prop:opd-coprod}~$\mylift\spsh\Env \co \OPD\to\RIG$, 
\ref{prop:free-coprod}~$\mylift\spsh\freesmc \co \CAT\to\RIG$ are cocartesian.
\end{proof}

\subsection*{Coexponentiability of symmetric 2-rigs} 
\label{sec:coexp}

The purpose of this section is to establish the following characterization.

\begin{thm}
\label{thm:coexponentiability-in-crig}
A symmetric 2-rig  is coexponentiable in $\RIG$ if and only if its underlying
cocomplete category is dualizable in $\COC$.  
\end{thm}

Using \cref{thm:dualizable-coc}, we get the following concrete reformulation.

\begin{cor} 
A symmetric 2-rig  is coexponentiable in $\RIG$ if and only if its underlying
cocomplete category is a deformation retract of a presheaf category over a
small category.
\end{cor} 

\Cref{thm:coexponentiability-in-crig} may be seen as a 2-dimensional version
of a result for rigs of Niefield and Wood~\cite{NiefieldS:coepr}.  Our proof
is a novel generalization of the method therein using the biadjunction
$\mylift\Sym \co \COC\rightleftarrows \RIG \co \mylift\sfU$
from~\cref{thm:sym-alg-coc}.  
The result of \cref{thm:coexponentiability-in-crig} will be a consequence of
\cref{prop:coexponentiability-in-crig:necessary,prop:coexponentiability-in-crig:sufficient}
which respectively prove that the condition is necessary and sufficient.

\begin{prop}
\label{prop:coexponentiability-in-crig:necessary}
If a symmetric 2-rig $\rigR$ is coexponentiable in $\RIG$, then the cocomplete
category $\mylift\sfU\rigR$ is dualizable in $\COC$.
\end{prop}
\begin{proof}
We start by showing that $\mylift\sfU(\rigR)\otimes-$ preserves bilimits of
$\lambda$-small diagrams.

Note first that the unit of the adjunction 
$\mylift\Sym \co \COC\rightleftarrows \RIG:\mylift\sfU$ is a pseudo-section,
providing a pseudo-natural deformation retract
\begin{equation}
  C \rightleftarrows \mylift\sfU\mylift\Sym(C)
  \mathrlap.
\end{equation}
Indeed, recalling from \cref{sec:2rig-pseudomonoids} that
$\mylift\Sym(C) = \oplus_n \mylift\Sym_n(C)$ and using the fact that countable
coproducts are also products in $\COC$ (\cref{lem:coc-additive}), the
pseudo-section arises from a projection map as the composite
$\mylift\sfU\mylift\Sym(C) 
 \simeq 
 \oplus_n \mylift\sfU\mylift\Sym_n(C) 
 \to 
 \mylift\sfU\mylift\Sym_1(C) 
 \simeq 
 C$.

Now, for a $\lambda$-small diagram $C_\bullet \co I\to\COC$, one shows that
the canonical map in $\COC$ 
\begin{equation*}\textstyle
\label{eq:map-retract-1}
\mylift\sfU(R)\otimes \bilim_i C_i
\to 
\bilim_i (\mylift\sfU(R)\otimes C_i)
\end{equation*}
is an equivalence by observing that it is a pseudo-retract of the following
canonical map 
\begin{equation*}\textstyle
\mylift\sfU(R)\otimes \bilim_i \mylift\sfU\mylift\Sym(C_i)
\to 
\bilim_i \big(\mylift\sfU(R) \otimes \mylift\sfU\mylift\Sym(C_i)\big)
\end{equation*}
which, in turn, is a canonical equivalence as follows:
\begin{align*}
\textstyle
\mylift\sfU(R)\otimes \bilim_i \mylift\sfU\mylift\Sym(C_i)
& \textstyle
\simeq \mylift\sfU(R)\otimes \mylift\sfU\big(\bilim_i \mylift\Sym(C_i)\big)
\\
& \textstyle
\simeq \mylift\sfU\big(R + \bilim_i \mylift\Sym(C_i)\big)
\\
& \textstyle
\simeq \mylift\sfU\bilim_i\big(R + \mylift\Sym(C_i)\big)
\\
& \textstyle
\simeq \bilim_i \mylift\sfU \big(R + \mylift\Sym(C_i)\big)
\\
& \textstyle
\simeq \bilim_i \big(\mylift\sfU(R)\otimes \mylift\sfU\mylift\Sym(C_i)\big)
\mathrlap{.} 
\end{align*}

Finally, the canonical strength map 
$\hom\COC{\mylift\sfU\rigR}{\Set}\otimes\mylift\sfU\rigR 
 \to 
 \hom\COC{\mylift\sfU\rigR}{\mylift\sfU\rigR}$ 
is an equivalence as follows: for any presentation of $\mylift\sfU\rigR$ as a
$\lambda$-small bicolimit of presheaf categories over $\kappa$-small
categories, $\bicolim_i\psh{A_i}$ (\cref{lem:ColimitPresentationForCOC}),
\begin{align*}
\hom\COC{\mylift\sfU\rigR}{\Set}\otimes\mylift\sfU\rigR
&\textstyle
\simeq \hom\COC{\bicolim_i\psh{A_i}}{\Set}\otimes\mylift\sfU\rigR
\\
&\textstyle
\simeq \big(\bilim_i\hom\COC{\psh{A_i}}{\Set}\big)\otimes\mylift\sfU\rigR
\\
&\textstyle
\simeq \bilim_i\big(\hom\COC{\psh{A_i}}{\Set}\otimes\mylift\sfU\rigR\big)
\\
&\textstyle
\simeq \bilim_i\hom\COC{\psh{A_i}}{\mylift\sfU\rigR} 
&\text{(by \cref{thm:dualizable-coc} and \cref{prop:DualInSMCBiCs})}
\\
&\textstyle
\simeq \hom\COC{\bicolim_i\psh{A_i}}{\mylift\sfU\rigR}
\\
&\simeq \hom\COC{\mylift\sfU\rigR}{\mylift\sfU\rigR}
\mathrlap{.}
\end{align*}
Thus, $\mylift\sfU\rigR$ is nuclear and, by \cref{thm:DualizableEqualNuclear},
dualizable.
\end{proof}

\begin{lem}
\label{lem:coexp-sym}
If a symmetric 2-rig $\rigR$ is such that the cocomplete category
$\mylift\sfU\rigR$ is dualizable in $\COC$ then, for every symmetric-algebra
2-rig $S$, the coexponential $\rigR\coexp\!\rigS$ in $\RIG$ exists and, for
every cocomplete category $C$ in $\COC$, the coexponential 
$\rigR\coexp \mylift\Sym(C)$ may be taken as the symmetric-algebra 2-rig 
$\mylift\Sym\big((\mylift\sfU\rigR)\dual\otimes C\big)$. 
\end{lem}
\begin{proof}
It follows from the equivalences: for all cocomplete categories $C$ and 2-rigs
$R, T$, 
\begin{align*}
\hom\RIG {\mylift\Sym(C)}{R+T} 
&\simeq \hom\COC {C}{\mylift\sfU (R+T)} \\
&\simeq \hom\COC{C}{{\mylift\sfU(R)} \otimes {\mylift\sfU(T)}} \\
&\simeq \hom\COC {(\mylift\sfU R)\dual\otimes C}{\mylift\sfU\,T} \\
&\simeq \hom\RIG {\mylift\Sym\big((\mylift\sfU R)\dual\otimes C\big)}{T} 
\mathrlap{.} 
\qedhere
\end{align*}
\end{proof}

\begin{prop}
\label{prop:coexponentiability-in-crig:sufficient}
If a symmetric 2-rig $\rigR$ is such that the cocomplete category
$\mylift\sfU\rigR$ is dualizable in $\COC$, then $\rigR$ is coexponentiable in
$\RIG$.  
\end{prop}
\begin{proof}
By~\cref{lem:coexp-sym}, we know that the coexponential $\rigR\coexp\!\rigS$
exists for $\rigS$ a symmetric-algebra 2-rig.
Let us see that this implies the existence of $\rigR\coexp\!\rigT$ for an
arbitrary 2-rig $\rigT$.  Indeed, for the natural presentation of $\rigT$ as a
bicolimit $\bicolim_i \rigS_i$ of symmetric-algebra 2-rigs $\rigS_i$ of
\cref{lem:sym-colim}, the equivalences: 
\begin{align*}
\textstyle
\hom\RIG{\bicolim_i\rigS_i}{\rigR+ - }
& \textstyle
\simeq \bilim_i \hom\RIG{ \rigS_i}{\rigR+ - }
\\
& \textstyle
\simeq \bilim_i \hom\RIG{ \rigR\coexp \rigS_i}{ - }
\\
&\textstyle
\simeq \hom\RIG{\bicolim_i\big(\rigR\coexp \rigS_i\big)}{ - }
\end{align*}
show that one may take 
$\rigR\coexp\!\rigT = \bicolim_i \big(\rigR\coexp \rigS_i\big)$.
\end{proof}

We conclude the section with some examples of coexponentiable 2-rigs.

\begin{ex}\leavevmode
\label{CoexpRigsExamples}
\begin{enumerate}
\item
Every essential subtopos of a topos of presheaves over a small category is a
coexponentiable cartesian 2-rig.  In particular, there are coexponentiable
cartesian 2-rigs that are not presheaf categories.

\item \label{DualSets}
The two-element set $\{ 0,1 \}$ has exactly two commutative monoid structures
with unit $0$, and multiplications $\cdot_0$ and $\cdot_1$ satisfying
$1\cdot_0 1 = 0$ and $1\cdot_1 1 = 1$.  These give rise to two discrete
symmetric monoidal categories 
respectively inducing two coexponentiable convolution symmetric 2-rig
structures on $\Set^{\{0,1\}}$
with unit $\one \defeq (1,0)$, 
and tensor products
\begin{align*}
  X\otimes_0 Y 
  &\defeq 
  (\, X_0\times Y_0 + X_1\times Y_1 \,,\, X_1\times Y_0 + X_0\times Y_1 \,)
  \mathrlap,
  \\
  X\otimes_1 Y 
  &\defeq 
  (\, X_0\times Y_0 \,,\, X_1\times Y_0 + X_0\times Y_1 + X_1\times Y_1 \,)
  \mathrlap.
\end{align*}

Besides the above, $\Set^{\{0,1\}}$ admits a further coexponentiable symmetric
2-rig structure with tensor unit~$\one$.  This arises in connection to the
categorification of commutative algebra as the \emph{2-rig of dual sets}
${\Set[\varepsilon]}$, with tensor product
\[
  X\otimes Y 
  \defeq  
  (\, X_0\times Y_0 \,,\, X_1 \times Y_0 + X_0 \times Y_1 \,)
  \mathrlap.
\]

The restriction functor $\imath^* \co \Set^{\mathfrak S} \to \Set^{\{0,1\}}$
in $\COC$ induced by the canonical embedding 
${\imath \co \{0,1\} \embto \mathfrak{S}}$ lifts to a morphism
$\mylift\spsh\freesmc\terminalcat \to \Set[\varepsilon]$ in $\RIG$; which, in
turn, is universally induced by the dual set
$\varepsilon \defeq 
(0,1) \in \Set[\varepsilon]$.  
Note thus that, in
$\Set[\varepsilon]$, $Z \cong Z_0\cdot\one + Z_1\cdot\varepsilon$ and 
\[
  (X_0\cdot\one+X_1\cdot\varepsilon)\otimes(Y_0\cdot\one+Y_1\cdot\varepsilon)
  \cong
  (X_0\times Y_0)\cdot\one + (X_1 \times Y_0 + X_0 \times Y_1)\cdot\varepsilon
  \mathrlap.
\]
In particular, $\varepsilon$ is a nilpotent dual set of 
degree~2.  In fact, the evaluation at $\varepsilon$ functor 
\[
  \RIG(\Set[\varepsilon],R) \to R
  \co F \mapsto F(\varepsilon)
\]
restricts to a natural equivalence between $\RIG(\Set[\varepsilon],R)$ and the
full subcategory of $R$ determined by the nilpotent objects of degree~2.
Hence, $\varepsilon\in\Set[\varepsilon]$ is a universal symmetric 2-rig with a
nilpotent object of degree~2.
\end{enumerate}
\end{ex}

\section{Cartesian closed bicategories} 
\label{sec:carcbs}

We draw some consequences concerning cartesian closed structure for symmetric
2-rigs.

\begin{prop} \label{prop:MixedCoexpFormula}
$\FRig$ is a coexponential ideal of $\CRig$.  

In fact, for a small symmetric monoidal category $M$ and a category $B$, 
\[
\mylift\spsh(M) \coexp \mylift\spsh(\freesmc B) 
\,\simeq\,
\mylift\spsh\freesmc\big(\sfU M^\op \times B\big)
\mathrlap.
\]
\end{prop}
\begin{proof}
Follows from the equivalences: 
\begin{align*}
\mylift\spsh(M) \coexp \mylift\spsh(\freesmc B)
&\simeq
\mylift\spsh(M) \coexp \mylift\Sym(\spsh B)
&&\text{(by Propositions~\ref{prop:pseudomonadicitysquare} 
  and~\ref{sec:2rig-pseudomonoids})}
\\
&\simeq
\mylift\Sym\big((\mylift\sfU\mylift\spsh M)\dual\otimes\spsh B\big)
&&\text{(by \cref{lem:coexp-sym})} 
\\
&\simeq
\mylift\Sym\big((\spsh\sfU M)\dual\otimes\spsh B\big)
&&\text{(by \cref{lem:coexp-sym})} 
\\
&\simeq
\mylift\Sym(\spsh\big(\sfU M^\op\big)\otimes\spsh B)
&&\text{(by \cref{thm:dualizable-coc})}
\\
&\simeq
\mylift\Sym\spsh\big(\sfU M^\op\times B\big)
&&\text{(by \cref{prop:coprod}~\ref{prop:conv-coprod})}
\\
&\simeq
\mylift\spsh\freesmc\big(\sfU M^\op \times B\big)
&&\text{(by Propositions~\ref{prop:pseudomonadicitysquare} 
  and~\ref{sec:2rig-pseudomonoids})}
\mathrlap{.}
\qedhere
\end{align*}
\end{proof}

The following are then~\cite[Theorem~5.4]{FioreM:carcbg}
and~\cite[Theorem~3.4.2]{GambinoN:opebaf} reinterpreted in terms of 2-rigs.
\begin{cor}
\label{cor:CC-FRIG}
The opposite of the bicategory $\FRig$ is cartesian closed.
\end{cor}
\begin{proof}
Because, by \cref{prop:MixedCoexpFormula}, 
\[
\mylift\spsh\freesmc(A) \coexp \mylift\spsh\freesmc(B) 
\ \simeq \
\mylift\spsh\freesmc\big(\sfU\freesmc A^\op \times B\big)
\]
for all categories $A$ and $B$.
\end{proof}
Note that the above does not extend to the whole of $\FRIG$ because not every
such 2-rig has an underlying cocomplete category dualizable in $\COC$.  

\bigskip
The following is~\cite[Theorem 5.4.6]{GambinoN:opebaf} reinterpreted in terms
of 2-rigs.

\begin{prop}
\label{prop:CC-OpRIG}
The opposite of the bicategory $\OpdRig$ is cartesian closed.	
\end{prop}
\begin{proof}
Using the biequivalence $\OpdRig \simeq \Bim(\FRig)$ proved in
\cite[Proposition~3.2.3 and Theorem~4.3.2]{AnelM:ope2r}, this result can be
viewed as a consequence of the general statement that the bimodule completion
of a cartesian closed bicategory is cartesian
closed~\cite[Theorem~5.1.2]{GambinoN:opebaf}.  
\end{proof}

We do not know whether the opposite of the cocartesian bicategory $\CRig$ is
cartesian closed, though we conjecture that it is not.  However, we have the
following partial result showing that it is closed under coexponentiation by
cartesian convolution 2-rigs.

\begin{prop}
For a small cartesian monoidal category $M$ and a monoidal category $N$,
\[
\mylift\spsh(M) \coexp \mylift\spsh(N) 
\simeq\,
\mylift\spsh( \freesmc(\sfU M^\op)\otimes N )
\mathrlap.
\]
\end{prop}
\begin{proof}
Let $M$ be a small cartesian monoidal category and let $N$ be a monoidal
category.

Note first the following equivalences
\begin{equation}\label{eqn:CRigcoexp}
\hom\RIG {\mylift\spsh(M) \coexp \mylift\spsh(N)} R
\,\simeq\,
\hom\RIG {\mylift\spsh(N)} {\mylift\spsh(M) + R}
\,\simeq\,
\Hom\SMCAT {N} {\mylift\sfUin(\mylift\spsh(M) + R)}
\mathrlap.
\end{equation}
Let us now analyze the symmetric monoidal category 
$\mylift\sfUin(\mylift\spsh(M) + R)$.  Up to equivalence, its underlying
category is 
\begin{equation}\label{eqn:CRigsum}
  \sfUin( \spsh(\sfU M)\otimes\mylift\sfU R )
  \,\simeq\,
  \hom\COC{\spsh(\sfU M^\op)}{\mylift\sfU R}
  \,\simeq\,
  \hom\CAT{\sfU M^\op}{\sfU\mylift\sfUin R}
\end{equation}
and, from the perspective of the latter description in terms of a functor
category, its symmetric monoidal structure arises from the cartesian monoidal
category $M$ by convolution; that is, it is given pointwise.  Thus, we have
the following equivalence of symmetric monoidal categories: 
\[
\mylift\sfUin(\mylift\spsh(M) + R)
\,\simeq\
\inthom\SMCAT{\freesmc(\sfU M^\op)}{\mylift\sfUin R}
\mathrlap.
\]
Therefore, from~\cref{eqn:CRigcoexp} and~\cref{eqn:CRigsum}, we have
\begin{align*}
\hom\RIG {\mylift\spsh(M) \coexp \mylift\spsh(N)} R
&\simeq\,
\Hom\SMCAT {N} { \inthom\SMCAT{\freesmc(\sfU M^\op)}{\mylift\sfUin R} }
\\
&\simeq\,
\Hom\SMCAT {\freesmc(\sfU M^\op)\otimes N} {\mylift\sfUin R} 
\\
&\simeq\,
\Hom\RIG {\mylift\spsh(\freesmc(\sfU M^\op)\otimes N)} {R} 
\mathrlap.
\qedhere
\end{align*}
\end{proof}

\end{document}